\documentclass{baustms}

\usepackage{amsmath}
\usepackage{amsthm}
\usepackage{amsfonts}
\usepackage{amssymb}

\newcommand\N{{\mathbf{N}}}

\renewcommand\P{{\mathcal{P}}}
\newcommand\eps{\varepsilon}

\renewcommand\Re{\operatorname{Re}}

\theoremstyle{bams}
\newtheorem{thm}{Theorem}[section]

\newtheorem{theorem}[thm]{Theorem}

\theoremstyle{bamsdefn}

\newtheorem{rem}[thm]{Remark}
\newtheorem{remark}[thm]{Remark}

\begin{document}

\title{A REMARK ON PARTIAL SUMS INVOLVING THE M\"OBIUS FUNCTION}

\cauthor 
\author[1]{Terence Tao}
\address[1]{Department of Mathematics, UCLA, Los Angeles CA 90095-1555\email{tao@math.ucla.edu}}

\authorheadline{T. Tao}

\support{The author is supported by NSF Research Award DMS-0649473, the NSF Waterman award and a grant from the MacArthur Foundation.  }

\begin{abstract}  Let $\langle \P \rangle \subset \N$ be a multiplicative subsemigroup of the natural numbers $\N = \{1,2,3,\ldots\}$ generated by an arbitrary set $\P$ of primes (finite or infinite).  We given an elementary proof that the partial sums $\sum_{n \in \langle \P \rangle: n \leq x} \frac{\mu(n)}{n}$ are bounded in magnitude by $1$.  With the aid of the prime number theorem, we also show that these sums converge to $\prod_{p \in \P} (1 - \frac{1}{p})$ (the case when $\P$ is all the primes is a well-known observation of Landau).  Interestingly, this convergence holds even in the presence of non-trivial zeroes and poles of the associated zeta function $\zeta_\P(s) := \prod_{p \in \P} (1-\frac{1}{p^s})^{-1}$ on the line $\{ \Re(s)=1\}$.

As equivalent forms of the first inequality, we have $|\sum_{n \leq x: (n,P)=1} \frac{\mu(n)}{n}| \leq 1$, $|\sum_{n|N: n \leq x} \frac{\mu(n)}{n}| \leq 1$, and $|\sum_{n \leq x} \frac{\mu(mn)}{n}| \leq 1$ for all $m,x,N,P \geq 1$.  
\end{abstract}



\maketitle

\section{Introduction}

Let $\N := \{1,2,\ldots\}$ be the natural numbers, and let $\mu: \N \to \{-1,0,+1\}$ be the M\"obius function, thus $\mu(n) = (-1)^k$ when $n$ is the product of $k$ distinct primes, and $\mu(n)=0$ otherwise.  Landau\cite{landau} made the elementary observation that the prime number theorem\footnote{We adopt the convention that $n$ always ranges over the natural numbers, and $p$ over prime numbers, unless otherwise stated.  The notation $o(1)$ denotes any quantity which converges to zero as $x \to \infty$, holding all other parameters fixed.}  
$$ 
\sum_{p \leq x} 1 = (1 + o(1)) \frac{x}{\log x}
$$
is equivalent to the conditional convergence of the infinite sum $\sum_n \frac{\mu(n)}{n} = 0$, thus
\begin{equation}\label{mun}
\sum_{n \leq x} \frac{\mu(n)}{n} = o(1).
\end{equation}
As is well known, \eqref{mun} and the prime number theorem are also both equivalent to the fact that the Riemann zeta function 
\begin{equation}\label{zeta}
\zeta(s) := \sum_n \frac{1}{n^s} = \prod_p (1-\frac{1}{p^s})^{-1}.
\end{equation}
has a simple pole at $s=1$ and no zeroes or poles elsewhere on the line $\{ \Re(s)=1\}$.
See \cite{diamond} for further discussion of this and other equivalences.

On the other hand, one has the elementary bound
\begin{equation}\label{el}
|\sum_{n \leq x} \frac{\mu(n)}{n}| \leq 1
\end{equation} 
for all $x$.  Indeed, to see this we may assume without loss of generality that $x$ is a natural number, and then sum the M\"obius inversion formula\footnote{We write $1_E$ to denote the indicator of a statement $E$, thus $1_E=1$ when $E$ is true and $1_E=0$ otherwise.} $1_{n=1} = \sum_{d|n} \mu(d)$ from $1$ to $x$ to obtain the identity
$$ 1 = \sum_{d \leq x} \mu(d) \lfloor \frac{x}{d}\rfloor = \sum_{d \leq x} \mu(d) \frac{x}{d} - \sum_{d < x} \mu(d) \{ \frac{x}{d} \}$$
where $\{y\} := y - \lfloor y \rfloor$ is the fractional part of $y$.  Using the trivial bound $|\mu(d) \{ \frac{x}{d} \}| \leq 1$ and the triangle inequality one obtains \eqref{el}.  The bound \eqref{el} is of course attained with equality when $x=1$.

In this paper we investigate the analogue of these facts when we ``turn off'' some of the primes in $\N$.  More precisely, we consider an arbitrary set $\P$ of primes (either finite or infinite), and let $\langle \P \rangle \subset \N$ be the multiplicative semigroup generated by $\P$ (i.e. the set of natural numbers whose prime factors all lie in $\P$).  We will study the behaviour of the sum
\begin{equation}\label{motor}
\sum_{n \in \langle \P \rangle: n \leq x} \frac{\mu(n)}{n}
\end{equation}
in $x$ and $\P$.  The analogue of the Riemann zeta function \eqref{zeta} is then the \emph{Burgess zeta function} $\zeta_\P$, defined for $\Re(s) > 1$ by the formula
\begin{equation}\label{zetap}
\zeta_\P(s) := \sum_{n \in \langle \P \rangle} \frac{1}{n^s} = \prod_{p \in \P} (1-\frac{1}{p^s})^{-1}.
\end{equation}

Note that as $\P$ is arbitrary, there need not be any asymptotic formula for the prime counting function $\sum_{p < x} 1$.  For similar reasons, there need not be any meromorphic continuation of $\zeta$ beyond the region $\{ \Re(s) > 1 \}$; for instance, one can easily construct a set $\P$ for which $\zeta_\P(s)$ blows up as $s \to 1^+$ at an intermediate rate between $1$ and $\frac{1}{s-1}$, which is not consistent with any meromorphic continuation at $s=1$.

Related to this, the zeta function $\zeta_\P(s)$ can develop zeroes or singularities on the line $\Re(s)=1$.  Indeed, observe for $\Re(s) > 1$ that
$$ \log|\zeta_\P(s)| = - \sum_{p \in \P} \log |1 - \frac{1}{p^s}| = \sum_{p \in \P} \frac{1}{p^s} + O(1).$$
For any non-zero real number $t$, if one sets $\P$ to be those primes $p$ for which $\{ \frac{t \log p}{2\pi} \} \leq 0.1$ (say), then one can easily check that $|\zeta_\P(1+it+\eps)| \to \infty$ as $\eps \to 0$; similarly, if we set $\P$ instead to be those primes for which $\{ \frac{t \log p}{2\pi} - \frac{1}{2} \} \leq 0.1$, then $|\zeta_\P(1+it+\eps)| \to 0$.  A modification of these examples shows that $\zeta_\P$ need not have a meromorphic continuation at $1+it$ for a fixed $t$, and with a bit more effort one can concoct a $\P$ for which $\zeta_\P$ has no continuation at $1+it$ for \emph{any} $t$; we omit the details.

Despite this, one can\footnote{Note added in proof: as pointed out to us after the submission of this article, these generalisations were essentially contained in \cite{gran} and \cite{skalba} respectively, see Remarks \ref{gran-rem} and \ref{skal-rem}.  We hope however that this article continues to serve an expository role in highlighting these elementary results.} generalise the statements \eqref{mun}, \eqref{el} to arbitrary $\P$.  
We first prove the generalisation of \eqref{el}, which is surprisingly elementary:

\begin{theorem}[Elementary bound]\label{mon}  For any $\P$ and $x$, one has
\begin{equation}\label{monster}
|\sum_{n \in \langle \P \rangle: n \leq x} \frac{\mu(n)}{n}| \leq 1.
\end{equation}
\end{theorem}

\begin{proof}  We may assume of course that $x$ is a natural number.  We may also assume that
\begin{equation}\label{sop}
\sum_{n \in \langle \P \rangle: n \leq x} \frac{1}{n} > 1.
\end{equation}
since the claim is immediate from the triangle inequality otherwise.

Let $\P'$ be the set of primes not in $\P$.  From M\"obius inversion one has
$$ 1_{n \in \langle \P'\rangle} = \sum_{d \in \langle \P\rangle: d|n} \mu(d)$$
for all natural numbers $n$; summing this over all $n \leq x$ as in the proof of \eqref{el} yields
$$ \sum_{n \in \langle \P'\rangle: n \leq x} 1 = \sum_{d \in \langle \P\rangle: d \leq x} \mu(d) \frac{x}{d} - \sum_{d \in \langle \P\rangle: d \leq x} \mu(d) \left \{ \frac{x}{d} \right \}.$$
Using the bound
$$ |\mu(d) \left \{ \frac{x}{d} \right \}| \leq 1 - \frac{1}{d}$$
we conclude that
\begin{equation}\label{zorn}
 |x \sum_{n \in \langle \P \rangle: n \leq x} \frac{\mu(n)}{n}| \leq 
\sum_{n \in \langle \P'\rangle: n \leq x} 1  + \sum_{n \in \langle \P\rangle: n \leq x} 1  - \sum_{n \in \langle \P\rangle: n \leq x} \frac{1}{n}.
\end{equation}
 Since $\langle \P \rangle$ and $\langle \P' \rangle$ overlap only at $1$, the claim now follows from \eqref{sop}.
\end{proof}

\begin{remark}\label{gran-rem} A bound very similar to \eqref{monster} was also observed in \cite{gran}.  In particular, the following refinement 
$$ x \sum_{n \in \langle \P \rangle: n \leq x} \frac{\mu(n)}{n} = \sum_{n \in \langle \P'\rangle: n \leq x} 1 + (1-\gamma) \sum_{n \in \langle \P\rangle: n \leq x} \mu(n) + O( \frac{x}{\log^{1/5} x} )$$
to \eqref{zorn} was obtained as a special case of \cite[Theorem 3.1]{gran}.  The lower bound of $-1$ for
$\sum_{n \in \langle \P \rangle: n \leq x} \frac{\mu(n)}{n}$ was also improved in \cite[Theorem 2]{gran} to
$$ (1 - 2 \log(1+\sqrt{e}) + 4 \int_1^{\sqrt{e}} \frac{\log t}{t+1} dt) \log 2 + o(1) = -0.4553\ldots + o(1),$$
which is optimal except for the $o(1)$ term, with a characterisation of those primes $\P$ for which the lower bound is attained.
\end{remark}

As corollaries of Theorem \ref{mon} we have
\begin{equation}\label{mock-1}
|\sum_{n \leq x: (n,P)=1} \frac{\mu(n)}{n}| \leq 1
\end{equation}
and
\begin{equation}\label{mock-2}
|\sum_{n|N: n \leq x} \frac{\mu(n)}{n}| \leq 1
\end{equation}
for any $P,x,N \geq 1$; also, from the identity $\mu(mn) =  \mu(m) \mu(n) 1_{(m,n)=1}$ one has
\begin{equation}\label{mock-3}
|\sum_{n \leq x} \frac{\mu(mn)}{n}| \leq 1
\end{equation}
for any $m,x \geq 1$.  These inequalities, which save a factor of $O(\log x)$ over the trivial bound, may be of some value in obtaining effective estimates in sieve theory or in exponential sums over the primes.

Now we turn to the generalisation of \eqref{mun}.

\begin{theorem}[Landau's theorem for arbitrary sets of primes]\label{main}  Let $\P$ be a set of primes.  Then the sum $\sum_{n \in \langle \P \rangle} \frac{\mu(n)}{n}$ converges conditionally to $\prod_{p \in \P} (1-\frac{1}{p})$, thus
\begin{equation}\label{monster2}
\sum_{n \in \langle \P \rangle: n \leq x} \frac{\mu(n)}{n} = \prod_{p \in \P} (1-\frac{1}{p}) + o(1)
\end{equation}
for all $x > 0$, where the decay rate of the error $o(1)$ depends on $\P$.  In particular, $\sum_{n \in \langle P\rangle} \frac{\mu(n)}{n}$ is conditionally convergent to zero if and only if $\sum_{p \in \P} \frac{1}{p}$ is infinite.
\end{theorem}

The proof of Theorem \ref{main} is also elementary (except for its use of \eqref{mun}, which is of course the special case of \eqref{monster2} when $\P$ consists of all the primes).  Interestingly, it is surprisingly difficult to replicate this elementary proof by zeta function methods, in large part due to the lack of meromorphic continuation alluded to earlier.  

\begin{remark}  A classical result of Wirsing (see e.g. \cite{hildebrand}) on mean values of multiplicative functions implies in particular that
$$ \frac{1}{x} \sum_{n \in \langle \P \rangle: n \leq x} \mu(n) = o(1).$$
This fact is also deducible from \eqref{monster2}.
\end{remark}

\begin{remark}\label{skal-rem} Theorem \ref{main} was also proven in \cite{skalba} by a similar method; in fact, the result in \cite{skalba} extends to arbitrary collections of prime ideals in algebraic number fields.
\end{remark}

We remark that the decay rate $o(1)$ in \eqref{monster2} is not uniform in $\P$.  For instance, if one takes $\P$ to be all the primes $p$ between $\sqrt{x}$ and $x$, one sees from Mertens' theorems that
$$ \sum_{n \in \langle \P \rangle: n \leq x} \frac{\mu(n)}{n} = 1 - \sum_{\sqrt{x} \leq p \leq x} \frac{1}{p} = 1 - \log 2 + o(1)$$
and
$$ \prod_{p \in \P} (1-\frac{1}{p}) = \frac{1}{2} + o(1)$$
and so one can keep the error term $o(1)$ in \eqref{monster2} bounded away from zero even for arbitrarily large $x$, by choosing $\P$ depending on $x$.  More precise statements of this nature can be found in \cite{gran}. 

Note that the above results do not hold when $\langle \P \rangle$ is replaced by a more general subsemigroup $G$ of the natural numbers (in which the generators are not necessarily prime).  For instance\footnote{We thank Andrew Sutherland for this example.}, if $G$ is the semigroup generated by the semiprimes (the products of two primes), then $\mu$ is either $0$ or $1$ and it is not difficult to see that $\sum_{n \in G} \frac{\mu(n)}{n}$ diverges.  One reason for the bad behaviour of these sums is that the zeta function $\sum_{n \in G} \frac{1}{n^s}$ no longer has an Euler product.

It is also essential that $\P$ consist of natural numbers, rather than merely real numbers larger than $1$ (as is the case in Beurling prime models), as the equal spacing of the integers is used in an essential way.  For instance, the inequality \eqref{monster} fails when $\P = \{1.1,1.2,1.3\}$ and $x=1.3$ (we thank Harold Diamond for this example and observation).

We thank Melvyn Nathanson and Keith Conrad for corrections, Harold Diamond for comments, and Wladyslaw Narkiewicz, Mariusz Ska{\l}ba, Kannan Soundararajan for references.  We are also indebted to the anonymous referee for corrections and suggestions.

\section{Proof of main theorem}

We now establish Theorem \ref{main}.  Fix $\P$; we allow all implied constants in the asymptotic notation to depend on $\P$.

If $\sum_{p \in \P} \frac{1}{p}$ is finite, then from the monotone convergence theorem we see that
$$ \sum_{n \in \langle \P \rangle} \frac{1}{n} = \prod_{p \in \P} (1 + \frac{1}{p-1})$$
is absolutely convergent, and thus (by dominated convergence)
$$ \sum_{n \in \langle \P \rangle} \frac{\mu(n)}{n} = \prod_{p \in \P} (1 - \frac{1}{p})$$
is conditionally convergent, giving the claim.  Thus we shall assume that $\sum_{p \in \P} \frac{1}{p}$ is infinite, in which case our task is to show that 
\begin{equation}\label{sum}
 \sum_{n \in \langle \P \rangle: n \leq x} \frac{\mu(n)}{n}  = o(1).
\end{equation}

Let $\P' := \{ p: p \not \in \P \}$ be the complement of $\P$ in the primes.  Suppose that $\sum_{p \in \P'} \frac{1}{p}$ is also infinite, thus
$$ \prod_{p \in \P} (1-\frac{1}{p}) = \prod_{p \in \P'} (1-\frac{1}{p}) = 0.$$
From elementary sieve theory (the Legendre sieve), this implies that both $\langle \P \rangle$ and $\langle \P' \rangle$ have asymptotic density zero, thus the right-hand side of \eqref{zorn} is $o(x)$, and \eqref{sum} follows.

The last remaining case is when $\sum_{p \in \P'} \frac{1}{p}$ is finite.  In this case, we use M\"obius inversion to write
$$ 1_{n \in \langle \P\rangle} = \sum_{d \in \langle \P'\rangle: d|n} \mu(d)$$
and thus
$$ \sum_{n \in \langle \P \rangle: n \leq x} \frac{\mu(n)}{n} = \sum_{d \in \langle \P'\rangle: d \leq x} \frac{\mu(d)}{d} \sum_{m \leq x/d} \frac{\mu(dm)}{m}.$$
By \eqref{mock-3}, $\frac{\mu(d)}{d} \sum_{m \leq x/d} \frac{\mu(dm)}{m}$ is bounded in magnitude by $\frac{1}{d}$.  As $\sum_{p \in \P'} \frac{1}{p}$ is finite, $\sum_{d \in \langle \P' \rangle} \frac{1}{d}$ is absolutely convergent, so by dominated convergence it suffices to show that
$$ \sum_{m \leq x} \frac{\mu(dm)}{m} = o(1)$$
for each fixed $d$.  

Fix $d$.  If we take $\P_d$ to be the primes dividing $d$, we observe the M\"obius inversion identity
\begin{align*}
\mu(dm) &= \mu(d) \mu(m) 1_{(d,m)=1} \\
&= \sum_{n \in \langle \P_d \rangle: n|m} \mu(d) \mu(m/n)
\end{align*}
and so
$$ \sum_{m \leq x} \frac{\mu(dm)}{m} = \mu(d) \sum_{n \in \langle \P_d \rangle: n \leq x}\frac{1}{n} \sum_{l \leq x/n} \frac{\mu(l)}{l}.$$
Since $\sum_{n \in \langle \P_d \rangle} \frac{1}{n}$ is absolutely convergent, the claim now follows from \eqref{el}, \eqref{mun} and the dominated convergence theorem.

\begin{rem}  By a convexity argument, one sees from \eqref{monster} that $|\sum_{n \leq x} \frac{\mu(n) a(n)}{n}| \leq 1$ for every multiplicative function $a: \N \to [0,1]$ taking values between zero and one (note that the expression in absolute values is affine-linear in $a(p)$ for each prime $p$, and \eqref{monster} is the special case when the $a(p)$ take the extreme values of $0$ and $1$); see also \cite{gran}.  It is also not difficult to adapt the arguments in this section to show that $\sum_{n=1}^\infty \frac{\mu(n) a(n)}{n}$ converges conditionally to $\prod_p (1 - \frac{a(p)}{p})$; we leave the details to the interested reader.
\end{rem}

\end{document}